\documentclass{amsart}

\usepackage{amsmath, graphicx, cite, enumerate, comment}

\newtheorem{theorem}{Theorem}[section]

\newtheorem{lemma}[theorem]{Lemma}

\newtheorem{corollary}[theorem]{Corollary}

\theoremstyle{definition}

\theoremstyle{remark}
\newtheorem{remark}[theorem]{Remark}

\numberwithin{equation}{section}

\newcommand{\abs}[1]{\left|#1\right|}

\newcommand{\Rm}{\textup{Rm}}
\newcommand{\Ric}{\textup{Ric}}

\newcommand{\Vol}{\textup{Vol}}
\newcommand{\diam}{\textup{diam}}

\newcommand{\FS}{\textup{FS}}
\newcommand{\Tr}{\textup{Tr}}

\newcommand{\eval}[2]{\left. #1 \right|_{#2}}

\newcommand{\R}{\mathbb{R}}
\newcommand{\C}{\mathbb{C}}

\newcommand{\CP}{\mathbb{P}}

\newcommand{\D}[2]{\frac{\partial #1}{\partial #2}}
\newcommand{\ddbar}{\partial \bar{\partial}}

\begin{document}
\title[On the collapsing rate of KRF with finite-time singularity]{On the collapsing rate of the K\"ahler-Ricci flow with finite-time singularity}
\author{Frederick Tsz-Ho Fong}
\begin{abstract}
We study the collapsing behavior of the K\"ahler-Ricci flow on a compact K\"ahler manifold $X$ admitting a holomorphic submersion $X \xrightarrow{\pi} \Sigma$ where $\Sigma$ is a K\"ahler manifold with $\dim_\C \Sigma < \dim_\C X$. We give cohomological and curvature conditions under which the fibers $\pi^{-1}(z)$, $z \in \Sigma$ collapse at the optimal rate $\textup{diam}_t (\pi^{-1}(z)) \sim (T-t)^{1/2}$.
\end{abstract}
\maketitle
\section{Introduction}

Let $X$ be a compact connected K\"ahler manifold with $\dim_\C = n$. Suppose $(\Sigma, \omega_\Sigma)$ is a K\"ahler manifold with $\dim_\C \Sigma = n-r < n$ and $X \xrightarrow{\pi} \Sigma$ is a surjective holomorphic submersion. For each $z \in \Sigma$, we call $\pi^{-1}(z)$ a fiber based at $z$, which is a complex submanifold of $X$ with $\dim_\C = r$. By the classical Ehresmann's fibration theorem \cite{Eh51}, $X$ is a smooth fiber bundle over $\Sigma$ and in particular $\pi^{-1}(z)$'s are diffeomorphic. Nonetheless, the induced complex structure on each $\pi^{-1}(z)$ may vary. In the case where all fibers $\pi^{-1}(z)$'s are biholomorphic, a classical theorem due to Fischer-Grauert  \cite{FG} asserts that $X$ is a holomorphic fiber bundle over $\Sigma$.

In this short note, we study the K\"ahler-Ricci flow on $X$ defined by
\begin{equation}\label{KRF}
\D{\omega_t}{t} = -\Ric(\omega_t), \quad \eval{\omega}{t=0} = \omega_0.
\end{equation}
The K\"ahler class $[\omega_t]$ at time $t$ is precisely given by $[\omega_0] - tc_1(X)$. The maximal existence time $T$ of \eqref{KRF} is uniquely determined by a result of Tian-Zhang in \cite{TZ06}, namely
\[T = \sup\{t : [\omega_0] - t c_1(X) > 0\}.\]
In particular, if $c_1(X) \cdot [\pi^{-1}(z)] > 0$ for some $z \in \Sigma$, we must have $T < \infty$. In this article, we only focus on the case where $\eqref{KRF}$ encounters finite-time singularity at $T < \infty$. We will study the collapsing behavior of the K\"ahler-Ricci flow \eqref{KRF} starting with an initial K\"ahler class $[\omega_0]$ such that as $t \to T$ the K\"ahler class $[\omega_t]$ limits to:
\begin{equation}\label{Kahler_class}
[\omega_0] - T c_1(X) = [\pi^*\omega_\Sigma].
\end{equation}

We will prove the following result:
\begin{theorem}\label{main}
Let $X \xrightarrow{\pi} \Sigma$ be a surjective holomorphic submersion. Suppose the K\"ahler-Ricci flow \eqref{KRF} on $X$ encounters finite-time singularity at $T < \infty$ and the initial K\"ahler class $[\omega_0]$ satisfies \eqref{Kahler_class}, we have
\begin{align*}
& \Ric(\omega_t) \leq B\omega_0 \text{ for some uniform constant } B > 0\\
\Rightarrow & \quad C^{-1} (T-t)^{1/2} \leq \textup{diam}_{\omega_t} \pi^{-1}(z) \leq C(T-t)^{1/2} \text{ for any } t \in [0, T), z \in \Sigma,
\end{align*}
where $C$ is a uniform constant depending only on $n, r, \omega_0, \omega_\Sigma$ and $B$.
\end{theorem}

In the case where $X$ is a $\CP^r$-bundle over $\Sigma$ and $X \xrightarrow{\pi} \Sigma$ is the bundle map, the fiber-collapsing behavior of the K\"ahler-Ricci flow was studied in \cite{SW09, SSW11} and by the author in \cite{Fong1}. In \cite{SW09}, Song-Weinkove studied the case where $r = 1$ and $(\Sigma, \omega_\Sigma) = (\CP^{n-1}, \omega_{\FS})$ and the initial metric $\omega_0$ on $X$ is constructed by Calabi's Ansatz (a cohomogeneity-1 symmetry). In this case, the K\"ahler class $[\omega_0]$ can be expressed as $- a_0 [\Sigma_0] + b_0 [\Sigma_\infty]$ with $0 < a_0 < b_0$, where $[\Sigma_0]$ and $[\Sigma_\infty]$ denote the Poincar\'e duals of the zero and infinity sections respectively, and \eqref{Kahler_class} can be achieved by a suitable choice of $a_0$ and $b_0$.  It was shown that under the condition \eqref{Kahler_class} the K\"ahler-Ricci flow collapses the $\CP^1$-fiber. In \cite{Fong1}, the author extended this result to allow $\Sigma$ to be any K\"ahler-Einstein manifold and proved that the singularity must be of Type I modelled on $\C^{n-1} \times \CP^1$ when the initial metric has Calabi symmetry. In \cite{SSW11}, Song-Sz\'ekelyhidi-Weinkove generalized this fiber-collapsing result to $\CP^r$-bundles over a smooth projective variety $\Sigma$ and removed the symmetry assumption.

In \cite{SSW11}, it was proved that the diameters of fibers decay at the rate of at most $(T-t)^{1/3}$ which is sufficient in their proof of Gromov-Hausdorff convergence towards the base manifold. It is more desirable for the diameters of fibers decay at a rate $\sim (T-t)^{1/2}$ as far as rescaling analysis (as in \cite{Fong1}) is concerned. Rescaling analysis has been fundamental in the study of singularity formations of the Ricci flow. The Cheeger-Gromov limit obtained from a suitably rescaled and dilated sequence of the Ricci flow encodes crucial geometric information of the singularity region near the singular time. In case of Type I singularity, the rescaling factor of the metric tensor is uniformly equivalent to $(T-t)^{-1}$. If a fiber $\pi^{-1}(z)$ shrinks at a rate such that $C^{-1}(T-t)^{1/2} \leq \diam_{\omega_t} (\pi^{-1}(z)) \leq C(T-t)^{1/2}$, then after the $(T-t)^{-1}$-rescaling the diameter of the fiber $F$ remains bounded away from $0$ and $\infty$. It is conjectured in \cite{SSW11, SW_KRF} that the diameter decay can be improved to $O((T-t)^{1/2})$. In this article, we give a partially affirmative answer to this conjecture under the assumption that the Ricci curvature stays bounded along the flow.

It is interesting to note that in the very special case where $\Sigma$ is a point, the only ``fiber" is simply the whole manifold $X$ and \eqref{Kahler_class} says that $\omega_0$ is in the canonical class, i.e. $[\omega_0] = Tc_1(X)$. It is well-known by an unpublished result of Perelman (see \cite{SeT08}) that the diameter stays bounded along the normalized K\"ahler-Ricci flow, and equivalently, decays at a rate of at most $(T-t)^{1/2}$ along the unnormalized flow \eqref{KRF}. Note that the Ricci curvature bound is not needed in Perelman's work. The proof involves the use of the monotonicity of the $\mathcal{W}$-functional introduced in \cite{P1}.

We mainly focus on finite-time singularity in this article. The collapsing behavior of the K\"ahler-Ricci flow with long-time existence was studied by, for instance, Song-Tian in \cite{ST07} when $X$ is a minimal elliptic surface. When the base Riemann surface has genus greater than $1$, it was proved in \cite{ST07} that the nonsingular Calabi-Yau fibers collapse at the rate $\sim e^{-t}$ under the normalized K\"ahler-Ricci flow $\partial_t \tilde{\omega_t} = -\Ric(\tilde\omega_t) - \tilde\omega_t$.

\newpage
\section{Some estimates}
In this section, we give the estimates necessary to establish Theorem \ref{main}. The approach to establish the main result is motivated by the techniques adopted in \cite{ST07, Z09, To10} etc.

We rewrite the K\"ahler-Ricci flow \eqref{KRF} as a parabolic complex Monge-Amp\`ere equation in the same way as in \cite{ST07, SW09, SSW11} etc. We define a family of reference metrics $\hat{\omega}_t$ in the same K\"ahler class as $\omega_t$ by:
\begin{equation}
\hat{\omega}_t = \frac{1}{T}((T-t)\omega_0 + t\pi^*\omega_\Sigma).
\end{equation}
One can argue that $[\omega_t] = [\hat\omega_t]$ by observing that $[\omega_t] = [\omega_0] - t c_1(X)$ which is a linear path connecting $[\omega_0]$ and $[\pi^*\omega_\Sigma]$ as $t$ goes from $0$ to $T$ by our assumption \eqref{Kahler_class}. By the $\ddbar$-lemma, there exists a family of smooth functions $\varphi_t$ such that $\omega_t = \hat\omega_t + \sqrt{-1}\ddbar\varphi_t$. Let $\Omega$ be a volume form on $X$ such that
\begin{equation}\label{Omega}
\sqrt{-1}\ddbar\log\Omega = \D{\hat\omega_t}{t} = \frac{1}{T}(\pi^*\omega_\Sigma - \omega_0).
\end{equation}
Then it is easy to check that the K\"ahler-Ricci flow \eqref{KRF} is equivalent to the following complex Monge-Amp\`ere equation:
\begin{equation}\label{MAP}
\D{\varphi_t}{t} = \log \frac{(\hat\omega_t + \sqrt{-1}\ddbar\varphi_t)^n}{(T-t)^r \Omega}.
\end{equation}
In all the estimates below, we will denote $C > 0$ to be a uniform constant which depends only on $n, r, \omega_0, \omega_\Sigma$ and $B$, and may change from line to line. We first prove the following:
\begin{lemma}\label{ddt_phi:bbd1}
Given that $\Ric(\omega_t) \leq B\omega_0$ for some uniform constant $B > 0$, then along \eqref{MAP} there exists a uniform constant $C = C(n, r, \omega_0, \omega_\Sigma, B)$ such that
\begin{equation}
\D{\varphi_t}{t} \leq \frac{1}{[\omega_0]^n} \int_X \log \frac{\omega_t^n}{(T-t)^r \Omega} \omega_0^n + C.
\end{equation}
\end{lemma}
\begin{proof}
We let $G(x,y) : X \times X - \{x = y\} \to \R$ be a nonnegative Green's function with respect to $\omega_0$. Then we have for any $x \in X$,
\begin{align*}
\D{\varphi_t}{t}(x) - \frac{1}{[\omega_0]^n} \int_X \D{\varphi_t}{t} \omega_0^n & = -\frac{1}{[\omega_0]^n} \int_X G(x, \cdot) \Delta_{\omega_0}\left(\D{\varphi_t}{t}\right) \omega_0^n\\
& = -\frac{1}{[\omega_0]^n} \int_X G(x, \cdot) \Delta_{\omega_0} \log \frac{\omega_t^n}{(T-t)^r \Omega} \omega_0^n\\
& = \int_X G(x,\cdot) \Tr_{\omega_0} \left(\Ric(\omega_t) + \D{\hat\omega_t}{t}\right) \omega_0^n\\
& = \int_X G(x,\cdot) \left\{\Tr_{\omega_0} \Ric(\omega_t) + \frac{1}{T} \Tr_{\omega_0}(\pi^*\omega_\Sigma - \omega_0)\right\} \omega_0^n,
\end{align*}
where we have used \eqref{Omega} in the third and fourth steps. By our assumption $\Ric(\omega_t) \leq B\omega_0$, it is easy to see that
\begin{equation}
\Tr_{\omega_0} \Ric(\omega_t) + \frac{1}{T} \Tr_{\omega_0}(\pi^*\omega_\Sigma - \omega_0) \leq C.
\end{equation}
for some uniform constant $C = C(n, r, \omega_0, \omega_\Sigma, B) > 0$.

Therefore, we have for $x \in X$,
\[\D{\varphi_t}{t}(x) - \frac{1}{[\omega_0]^n} \int_X \D{\varphi_t}{t} \omega_0^n \leq C \int_X G(x, y) \omega_0^n(y).\]
Note that $\int_X G(x,y) \omega_0^n(y) \equiv \text{constant}$. It completes the proof of the lemma.
\end{proof}
Next we derive a uniform upper bound for $\D{\varphi_t}{t}$. By the previous lemma, it suffices to bound the total volume of $(X, \omega_t)$, which depends only on $[\omega_t]$. Let us consider the reference metric $\hat\omega_t = \frac{1}{T}((T-t)\omega_0 + t\pi^*\omega_\Sigma)$. We have
\begin{equation*}
\hat\omega_t^n = \frac{1}{T^n} \sum_{k = 0}^n {n \choose k} (T-t)^k t^{n-k} \omega_0^k \wedge (\pi^*\omega_\Sigma)^{n-k}.
\end{equation*}
Since $\Sigma$ has complex dimension $n-r$, we have $(\pi^*\omega_\Sigma)^{n-k} = 0$ for any $k < r$. Therefore, we have
\begin{align*}
\hat\omega_t^n & = \frac{1}{T^n} \sum_{k=r}^n {n \choose k} (T-t)^k t^{n-k} \omega_0^k \wedge (\pi^*\omega_\Sigma)^{n-k}\\
& = \frac{(T-t)^r}{T^n} \left\{ {n \choose r} t^{n-r} \omega_0^r \wedge (\pi^*\omega_\Sigma)^{n-r} + \ldots + (T-t)^{n-r} \omega_0^n\right\}.
\end{align*}
It is not difficult to see that there exists a uniform constant $C = C(n, r, \omega_0, \omega_\Sigma) > 0$ such that for any $t \in [0,T)$, we have
\begin{equation}\label{volume_form_hat}
C^{-1} (T-t)^r \omega_0^n \leq \hat\omega_t^n \leq C(T-t)^r \omega_0^n.
\end{equation}
By Jensen's Inequality, we have
\begin{align*}
\int_X \log\frac{\omega_t^n}{(T-t)^r \Omega} \frac{\omega_0^n}{[\omega_0]^n} & \leq \log \left(\int_X \frac{\omega_t^n}{(T-t)^r \Omega} \frac{\omega_0^n}{[\omega_0]^n}\right)\\
& \leq \log \left(\frac{C}{[\omega_0]^n} \int_X \frac{\hat\omega_t^n}{(T-t)^r}\right)\\
& \leq \log C \qquad \text{using \eqref{volume_form_hat}}.
\end{align*}
Combining with Lemma \ref{ddt_phi:bbd1}, we have established the following:
\begin{lemma}\label{ddt_phi:bbd2}
There exists a uniform constant $C = C(n, r, \omega_0, \omega_\Sigma, B) > 0$ such that for any $t \in [0, T)$, we have
\begin{equation}\label{volume_form_bbd}
\frac{\omega_t^n}{(T-t)^r \Omega} = e^{\D{\varphi_t}{t}} \leq C.
\end{equation}
\end{lemma}
Lemma \ref{ddt_phi:bbd2} gives a pointwise bound for the volume form, which will be used in the next lemma to show the K\"ahler potential $\varphi_t$ is decay at a rate of $O(T-t)$ after a suitable normalization.

For each $z \in \Sigma$ and $t \in [0,T)$, we denote $\omega_{t,z}$ to be the restriction of $\omega_t$ on the fiber $\pi^{-1}(z)$. For each $t \in [0,T)$, we define a function $\Phi_t : \Sigma \to \R$ by
\[\Phi_t (z) = \frac{1}{\Vol_{\omega_{0,z}}(\pi^{-1}(z))}\int_{\pi^{-1}(z)} \varphi_t ~ \omega_{0,z}^r\]
which is the average value of $\varphi_t$ over each fiber $\pi^{-1}(z)$. The pull-back $\pi^*\Phi_t$ is then a function defined on $X$. For simplicity, we also denote $\pi^*\Phi_t$ by $\Phi_t$.
\begin{lemma}\label{potential_bbd}
There exists a uniform constant $C = C(n, r, \omega_0, \omega_\Sigma, B)$ such that for any $t \in [0,T)$, we have
\begin{equation}
\abs{\frac{\varphi_t - \Phi_t}{T-t}} \leq C
\end{equation}
\end{lemma}
\begin{proof}
Denote $\tilde\varphi_t = \frac{\varphi_t - \Phi_t}{T-t}$. For each $z \in \Sigma$, we have
$\hat\omega_{t,z} = \frac{T-t}{T}\omega_{0,z}$, and so 
\[\omega_{t,z} = \frac{T-t}{T}\omega_{0,z} + \eval{\sqrt{-1}\ddbar\varphi_t}{\pi^{-1}(z)}.\]
Since $\Phi_t$ depends only on $z \in \Sigma$, we have $\eval{\sqrt{-1}\ddbar\Phi_t}{\pi^{-1}(z)} = 0$. By rearranging, we have
\begin{equation}\label{metric_restricted}
\frac{1}{T-t}\omega_{t,z} = \frac{1}{T}\omega_{0,z} + \eval{\sqrt{-1}\ddbar\tilde\varphi_t}{\pi^{-1}(z)}.
\end{equation}
Regard \eqref{metric_restricted} to be a metric equation on the manifold $\pi^{-1}(z)$, then we have
\begin{equation}\label{volume_r}
\left(\frac{1}{T}\omega_{0,z} + \eval{\sqrt{-1}\ddbar\tilde\varphi_t}{\pi^{-1}(z)}\right)^r = \left(\frac{1}{T-t} \omega_{t,z}\right)^r
\end{equation}
Using Lemma \ref{ddt_phi:bbd2}, we can see
\begin{align}\label{volume_r_bbd}
\frac{\omega_{t,z}^r}{\omega_{0,z}^r} & = \frac{\omega_t^r \wedge (\pi^*\omega_\Sigma)^{n-r}}{\omega_0^r \wedge (\pi^*\omega_\Sigma)^{n-r}}\\
\nonumber & \leq \frac{\omega_t^r \wedge (\pi^*\omega_\Sigma)^{n-r}}{\omega_t^n} \cdot \frac{\omega_t^n}{\omega_0^r \wedge (\pi^*\omega_\Sigma)^{n-r}}\\
\nonumber & \leq C(\Tr_{\omega_t}\pi^*\omega_\Sigma)^r \cdot (T-t)^r.
\end{align}
It is well-known that $\Tr_{\omega_t} \pi^*\omega_\Sigma \leq C$ by the cohomological condition \eqref{Kahler_class} and a maximum principle argument (see e.g. \cite{SW09, SSW11}). Combining this with 
\eqref{volume_r_bbd}, we see than \eqref{volume_r} can be restated as
\begin{equation}
\left(\frac{1}{T}\omega_{0,z} + \eval{\sqrt{-1}\ddbar\tilde\varphi_t}{\pi^{-1}(z)}\right)^r = F_z(\xi,t) \left(\frac{1}{T}\omega_{0,z}\right)^r
\end{equation}
where $F_z(\xi,t) : \pi^{-1}(z) \times [0,T) \to \R_{>0}$ is uniformly bounded.

Since $\int_{\pi^{-1}(z)} \tilde\varphi_t \omega_{0,z}^r = 0$, by applying Yau's $L^\infty$-estimate (see \cite{Y78}) on \eqref{metric_restricted}, we then have
\begin{equation}
\sup_{\pi^{-1}(z) \times [0,T)}|\tilde\varphi_t| \leq C_z,
\end{equation}
where $C_z$ depends on $n, r, \omega_0, \omega_\Sigma, B, \sup_{\pi^{-1}(z) \times [0,T)}F_z, \Vol_{\omega_{0,z}}(\pi^{-1}(z))$, the Sobolev and Poincar\'e constants of $\pi^{-1}(z)$ with respect to metric $\omega_{0,z}$, all of which can be bounded uniformly independent of $z$. It completes the proof of the lemma.
\end{proof}
\begin{remark}
In our setting, the uniform boundedness of Sobolev and Poincar\'e constants of $(\pi^{-1}(z), \omega_{0,z})$ follow from the compactness of $\Sigma$ and the absence of singular fibers. It is also possible to give such uniform bounds by noting that $\pi^{-1}(z)$'s are minimal submanifolds of $X$ and hence they can be embedded into some Euclidean space $\R^N$ with bounded mean curvature. Combining the classical results in \cite{MS73, Ch75, LY80} one can obtain uniform bounds on the Sobolev and Poincar\'e constants. A detail discussion in this regard can be found in \cite{To10} which also dealt with singular fibers.
\end{remark}

\section{Proof of Theorem \ref{main}}
\begin{proof}[Proof of Theorem \ref{main}]
We apply maximum principle to the following quantity
\[Q := \log((T-t)\Tr_{\omega_t}\omega_0) - \frac{A}{T-t}(\varphi_t - \Phi_t),\]
where $A$ is a positive constant to be chosen.
Denote $\square_{\omega_t} = \partial_t - \Delta_{\omega_t}$, we have
\begin{align}\label{metric_lower_bd:ineq1}
\square_{\omega_t} \log((T-t)\Tr_{\omega_t}\omega_0) & \leq -\frac{1}{T-t} + \frac{1}{\Tr_{\omega_t}\omega_0} g^{i\bar j} g^{k\bar l} \Rm(g_0)_{i\bar j k\bar l}\\
\nonumber & \leq -\frac{1}{T-t} + \tilde{C}\Tr_{\omega_t}\omega_0
\end{align}
where $\tilde{C}$ depends only on the curvature of $g_0$.
\begin{align*}
\square_{\omega_t} \frac{A}{T-t}(\varphi_t - \Phi_t) & = \frac{A}{T-t}\left(\D{\varphi_t}{t} - \D{\Phi_t}{t}\right) - \frac{A}{(T-t)^2}(\varphi_t - \Phi_t)\\
& \quad - \frac{A}{T-t}(\Delta_{\omega_t} \varphi_t - \Delta_{\omega_t} \Phi_t)\\
& \geq \frac{A}{T-t}\left(\log\frac{\omega_t^n}{(T-t)^r\Omega} - \int_{\pi^{-1}(z)} \D{\varphi_t}{t} \omega_{0,z}^r \right) - \frac{CA}{T-t}\\
& \quad - \frac{A}{T-t}(n-\Tr_{\omega_t}\hat\omega_t - \Delta_{\omega_t}\Phi_t)
\end{align*}
Note that
\[\Tr_{\omega_t}\hat\omega_t \geq n\left(\frac{\hat\omega_t^n}{\omega_t^n}\right)^{1/n}\geq cn\left(\frac{(T-t)^r\Omega}{\omega_t^n}\right)^{1/n}.\]
Since $x \mapsto -\log x + \frac{1}{2}cnx^{1/n}$ is uniformly bounded from below, we have
\begin{equation}
\log\frac{\omega_t^n}{(T-t)^r\Omega} + \frac{1}{2}\Tr_{\omega_t}\hat\omega_t \geq -\log\frac{(T-t)^r\Omega}{\omega_t^n} + \frac{1}{2}cn\left(\frac{(T-t)^r\Omega}{\omega_t^n}\right)^{1/n} \geq C^{-1}
\end{equation}
for some uniform constant $C > 0$. Hence, we have
\begin{align}\label{metric_lower_bd:ineq2}
\square_{\omega_t} \frac{A}{T-t}(\varphi_t - \Phi_t) & \geq -\frac{AC}{T-t}\\
\nonumber & \quad + \frac{A}{T-t} \left(\Delta_{\omega_t}\Phi_t-\int_{\pi^{-1}(z)} \D{\varphi_t}{t} \omega_{0,z}^r\right)\\
\nonumber & \quad + \frac{A}{2(T-t)}\Tr_{\omega_t}{\hat\omega_t}.
\end{align}
Combining \eqref{metric_lower_bd:ineq1} and \eqref{metric_lower_bd:ineq2}, we have
\begin{align}\label{metric_lower_bd:ineq3}
\square_{\omega_t} Q & \leq \frac{AC}{T-t} + \tilde{C}\Tr_{\omega_t}\omega_0 - \frac{A}{2(T-t)}\Tr_{\omega_t}\left(\frac{T-t}{T}\omega_0 + \frac{t}{T}\pi^*\omega_\Sigma\right)\\
\nonumber & \quad - \frac{A}{T-t} \left(\Delta_{\omega_t}\Phi_t-\int_{\pi^{-1}(z)} \D{\varphi_t}{t} \omega_{0,z}^r\right)\\
\nonumber & \leq \frac{AC}{T-t} + \left(\tilde{C} - \frac{A}{2T}\right)\Tr_{\omega_t}\omega_0\\
\nonumber & \quad - \frac{A}{T-t} \left(\Delta_{\omega_t}\Phi_t-\int_{\pi^{-1}(z)} \D{\varphi_t}{t} \omega_{0,z}^r\right)
\end{align}
By Lemma \ref{ddt_phi:bbd2}, we have $\D{\varphi_t}{t}$ for some uniform constant $C$. It follows that
\[\int_{\pi^{-1}(z)} \D{\varphi_t}{t} \omega_{0,z}^r \leq C\Vol_{\omega_{0,z}} (\pi^{-1}(z)) \leq C'.\]
Note that $\Vol_{\omega_{0,z}} (\pi^{-1}(z))$ is independent of $z$.
For the Laplacian term of $\Phi_t$, we have
\begin{align*}
\Delta_{\omega_t} \int_{\pi^{-1}(z)}\varphi_t \omega_{0,z}^r & = \Tr_{\omega_t}\int_{\pi^{-1}(z)} \sqrt{-1}\ddbar\varphi_t \wedge \omega_{0,z}^r\\
& = \Tr_{\omega_t}\int_{\pi^{-1}(z)} (\omega_t-\hat\omega_t) \wedge \omega_{0,z}^r\\
& \geq -\Tr_{\omega_t} \int_{\pi^{-1}(z)} \hat\omega_t \wedge \omega_{0,z}^r\\
& \geq -\Tr_{\omega_t}\int_{\pi^{-1}(z)}\left(\omega_0 \wedge \omega_{0,z}^r + \pi^*\omega_\Sigma \wedge \omega_{0,z}^r\right).
\end{align*}
By the fact that $\Tr_{\omega_t}\pi^*\omega_\Sigma \leq C$ and $\int_{\pi^{-1}(z)}\left(\omega_0 \wedge \omega_{0,z}^r + \pi^*\omega_\Sigma \wedge \omega_{0,z}^r\right)$ is a $(1,1)$-form on $\Sigma$ independent of $t$, we have 
\[\Delta_{\omega_t} \int_{\pi^{-1}(z)}\varphi_t \omega_{0,z}^r \geq -C.\]
for some uniform constant $C$. Back to \eqref{metric_lower_bd:ineq3}, we have
\begin{equation}
\square_{\omega_t} Q \leq \frac{AC}{T-t} + \left(\tilde{C}-\frac{A}{2T}\right)\Tr_{\omega_t}\omega_0 \leq \frac{AC}{T-t} - \Tr_{\omega_t}\omega_0
\end{equation}
if we choose $A$ sufficiently large such that $\tilde{C}-\frac{A}{2T} < -1$.

Hence, for any $\varepsilon > 0$, at the point where $Q$ achieves its maximum over $X \times [0,T-\varepsilon]$, we have $\Tr_{\omega_t}(T-t)\omega_0 \leq C$ for some uniform constant $C$ independent of $\varepsilon$. Together with Lemma \ref{potential_bbd}, it shows that for any $t \in [0,T)$ we have,
\begin{equation}\label{metric_lower_bbd:semifinial}
C^{-1}(T-t)\omega_0 \leq \omega_t.
\end{equation}
Combining this with the fact that $\omega_t \geq C^{-1}\pi^*\omega_\Sigma$, we have
\begin{equation}\label{metric_lower_bbd:final}
C^{-1}\hat\omega_t \leq \omega_t.
\end{equation}
Together with \eqref{volume_form_bbd} and \eqref{volume_r_bbd}, we also have $\omega_t \leq C\hat\omega_t$ for any $t \in [0,T)$. It completes the proof of the theorem since one can check that for the reference metric $\hat\omega_t$, we have
\[C^{-1}(T-t)^{1/2} \leq \diam_{\hat\omega_t}(\pi^{-1}(z)) \leq C(T-t)^{1/2}.\]
\end{proof}
\section{Remarks on curvatures}\label{sec:curvature}
We would like to end this article by a discussion of the implications of Theorem \ref{main} on the blow-up rate of curvatures.

There are ``folklore conjectures" concerning the blow-up of curvatures along the K\"ahler-Ricci flow \eqref{KRF} with finite-time singularity (see Section 7 in \cite{SW_KRF}). It is known by Hamilton \cite{H95}, Sesum \cite{Se05} and Zhang \cite{Z10} that $\sup \|\Rm\|_{g(t)}$, $\sup \|\Ric\|_{g(t)}$ and the scalar curvature $\sup R(g(t))$ must blow-up to $+\infty$ as $t \to T$ when $T < \infty$ is the singular time. However, it is still open whether they blow-up at the rate of $O((T-t)^{-1})$ for the K\"ahler-Ricci flow \eqref{KRF}.

In the case where $[\omega_0] = c_1(X) > 0$, it was established by Perelman (see \cite{SeT08}) that $R(g(t)) = O((T-t)^{-1})$ along \eqref{KRF}. For the normalized K\"ahler-Ricci flow $\partial_t \tilde{g}_t = -\Ric(\tilde{g}_t) - \tilde{g}_t$ with finite-time singularity, Zhang established in \cite{Z10} that $R(\tilde{g}(t)) = O((T-t)^{-2})$ under a cohomological assumption analogous to \eqref{Kahler_class}. In the special case where $X$ is a $\CP^1$-bundle over a compact K\"ahler-Einstein manifold, it is proved in \cite{Fong1} that when the $\CP^1$-fibers collapse the K\"ahler-Ricci flow \eqref{KRF} must develop Type I singularity (i.e. $\|\Rm\|_{g(t)} = O((T-t)^{-1})$) assuming that the initial metric has Calabi symmetry.

Under the cohomological setting in this article, the implications of the boundedness of $\Tr_{\omega_0}\Ric(\omega_t)$ on the curvature blow-up rates are as follows:
\begin{corollary}
Under the same assumptions as in Theorem \ref{main}, we have
\begin{align*}
& \Ric(\omega_t) \leq B\omega_0 \text{ for some uniform constant } B > 0\\
\Rightarrow & \quad R(\omega(t)) = O((T-t)^{-1}), \text{ and } \|\Rm\|_{\omega_t} = O((T-t)^{-2}).
\end{align*}
\end{corollary}
\begin{proof}
The $O((T-t)^{-1})$ blow-up rate of the scalar curvature $R(\omega_t) = \Tr_{\omega_t}\Ric(\omega_t)$ follows trivially from $\Ric(\omega_t) \leq B\omega_0$ and \eqref{metric_lower_bbd:semifinial}.

For the blow-up of the Riemann curvature tensor, we use the result in the proof of Theorem \ref{main} that there exists a uniform constant $C > 0$ such that for any $t \in [0,T)$, we have
\[C^{-1}(T-t)\omega_0 \leq \omega_t \leq C\omega_0.\]
Using the estimates in \cite{PSS07, ShW11} one can establish that 
\[\|\Rm\|_{\omega_t} = O((T-t)^{-2}).\]
Note that in \cite{ShW11} the authors assumed $N\omega_0 \leq \omega_t \leq \frac{1}{N}\omega_0$ and asserted that $\|\Rm\|_{\omega_t} = O(N^{-4})$. One can easily check this result can be extended without much difficulty so that $N$ is any positive non-increasing function $N(t)$ defined on $[0,T)$. Furthermore, if the metric upper bound $\frac{1}{N}\omega_0$ is replaced by $C\omega_0$ for some uniform constant $C > 0$, the result can be generalized to $\|\Rm\|_{\omega_t} = O(N^{-2})$ with an almost identical proof.
\end{proof}
\subsection*{Acknowledgements}
The author would like to thank his advisor Richard Schoen for his continuing encouragement and many inspiring ideas. He would also like to thank Simon Brendle and Yanir Rubinstein for many productive discussions and their interests in his work.
\bibliographystyle{amsalpha}
\bibliography{../citations}
\end{document}